\documentclass[12pt,reqno,english]{amsproc}
\usepackage[margin=0.95in]{geometry}
\usepackage{abstract, amsmath, amssymb, amsthm, array,
babel,bm,
catchfilebetweentags, cite, color,
dsfont,
emptypage,
fancyhdr, float, fnpos,
graphicx,
hyperref,
ifthen,
latexsym,
mathtools,
pdfpages,
qtree,
subcaption, 	
verbatim,
wrapfig}

\usepackage[all]{xy}

\hypersetup
{
    colorlinks,	
    citecolor=red,
    filecolor=black,
    linkcolor=blue,
    urlcolor=blue
}

\SelectTips{cm}{}

\DeclarePairedDelimiterX\setc[2]{\{}{\}}{\,#1 \;\delimsize\vert\; #2\,}

\DeclarePairedDelimiter\abs{\lvert}{\rvert}%
\DeclarePairedDelimiter\norm{\lVert}{\rVert}%

\makeatletter
\let\oldabs\abs
\def\abs{\@ifstar{\oldabs}{\oldabs*}}
\let\oldnorm\norm
\def\norm{\@ifstar{\oldnorm}{\oldnorm*}}
\makeatother

\newtheorem{theorem}{Theorem}[section]
\newtheorem{corollary}[theorem]{Corollary}

\theoremstyle{definition}
\newtheorem{definition}[theorem]{Definition}
\newtheorem{example}[theorem]{Example}
\newtheorem{remark}[theorem]{Remark}

\newtheoremstyle{algorithm} 
    {\topsep}                    
    {\topsep}                    
    {\ttfamily}                   
    {}                           
    {\scshape}                   
    {.}                          
    {.5em}                       
    {}  
\theoremstyle{algorithm}

\newtheoremstyle{theorem-w/o-number}
  {\topsep}   
  {\topsep}   
  {\itshape}  
  {0pt}       
  {\bfseries} 
  {}         
  {5pt plus 1pt minus 1pt} 
  {}          
\theoremstyle{theorem-w/o-number}
\newtheorem*{theorem*}{Theorem}
\newtheorem*{teorema*}{Teorema}
\newtheorem*{corollary*}{Corollary}
\newtheorem*{corolario*}{Corolario}
\newtheorem*{lemma*}{Lemma}
\newtheorem*{lema*}{Lema}
\newtheorem*{proposition*}{Proposition}
\newtheorem*{proposicion*}{Proposici\'on}
\newtheorem*{notation*}{Notation}
\newtheorem*{assumptions*}{Assumptions}

\newtheoremstyle{definition-w/o-number}
  {\topsep}   
  {\topsep}   
  {\normalfont}  
  {0pt}       
  {\bfseries} 
  {}         
  {5pt plus 1pt minus 1pt} 
  {}          
\theoremstyle{definition-w/o-number}
\newtheorem*{definition*}{Definition}
\newtheorem*{definicion*}{Definici\'on}
\newtheorem*{example*}{Example}
\newtheorem*{ejemplo*}{Ejemplo}\newtheorem*{remark*}{Remark}
\newtheorem*{observacion*}{Observaci\'on}

\numberwithin{equation}{section}

\let\tmp\oddsidemargin
\let\oddsidemargin\evensidemargin
\let\evensidemargin\tmp
\reversemarginpar

\newcolumntype{L}[1]{>{\raggedright\let\newline\\\arraybackslash\hspace{0pt}}m{#1}}
\newcolumntype{C}[1]{>{\centering\let\newline\\\arraybackslash\hspace{0pt}}m{#1}}
\newcolumntype{R}[1]{>{\raggedleft\let\newline\\\arraybackslash\hspace{0pt}}m{#1}}

\usepackage[font=footnotesize,labelfont=bf]{caption}

\newcommand{\Cech}{\v{C}}
\newcommand{\cech}{\check{\mathcal{C}}}

\newcommand{\bR}{\mathbb R}
\newcommand{\bS}{\mathbb S}
\newcommand{\bT}{\mathbb T}

\newcommand{\Image}{\text{\rm Im}\,}

\newcommand{\Ker}{\text{\rm Ker}\,}

 \newcommand{\dlie}{\partial }

\newcommand{\De}{\Delta}

\usepackage{tikz}
\usepackage{tikz-cd}

\usepackage{mathtools}
\usetikzlibrary{matrix,arrows}
\tikzset{commutative diagrams/diagrams={ampersand replacement=\&}}

\tikzset{dbl/.style={double,
		double equal sign distance,
		-implies,
		shorten >=10pt,
		shorten <=10pt}}

\tikzset{
	invisible/.style={opacity=0},
	visible on/.style={alt={#1{}{invisible}}},
	alt/.code args={<#1>#2#3}{%
		\alt<#1>{\pgfkeysalso{#2}}{\pgfkeysalso{#3}}%
	}
}

\newcommand{\X}{\mathbb{X}}

\newcommand{\CC}{\mathcal{C}}
\newcommand{\DD}{\mathcal{D}}

\renewcommand{\SS}{\mathcal{S}}

\newcommand{\Cfunc}{\mathbf{C}}

\newcommand{\Wfunc}{\mathbf{W}}
\newcommand{\Ffunc}{\mathbf{F}}
\newcommand{\Gfunc}{\mathbf{G}}

\newcommand{\Hfunc}{\mathbf{H}}

\newcommand{\Vfunc}{\mathbf{V}}

\newcommand{\Vect}{\mathbf{Vect}}

\newcommand{\Top}{\mathbf{Top}}

\usepackage[latin1]{inputenc}

\makeatletter
\def\l@subsection{\@tocline{2}{0pt}{2.5pc}{5pc}{}}
\makeatother

\providecommand{\keywords}[1]
{
  \small	
  \textbf{\textit{Keywords---}} #1
}

\begin{document}

\title{{$A_\infty$ persistent homology} estimates the topology from pointcloud datasets}

\author{Francisco Belch\'{\i}$^1$
    \footnote{\noindent $^1$ frbegu@gmail.com -
        \href{https://orcid.org/0000-0001-5863-3343}{https://orcid.org/0000-0001-5863-3343}\\
        Institut de Rob\`otica i Inform\`atica Industrial, CSIC-UPC \\
        Llorens i Artigas 4-6, 08028 Barcelona, Spain
    }, 
    Anastasios Stefanou$^2$
    \footnote{\noindent $^2$ stefanou.3@osu.edu\\
        Mathematical Biosciences Institute; Department of Mathematics \\
        The Ohio State University
        }
}

\maketitle

\tableofcontents

\begin{abstract}
Let $X$ be a closed subspace of a metric space $M$.
Under mild hypotheses, one can estimate the Betti numbers of $X$ from a finite set $P \subset M$ of points approximating $X$.
In this paper, we show that one can also use $P$ to estimate
much more detailed topological properties of $X$. These properties are computed via $A_\infty$-structures, and are therefore related to the cup and Massey products of $X$, its loop space $\Omega X$, its formality, linking numbers, etc.

Additionally, we study the following setting:
given a continuous function $f \colon Y \longrightarrow \mathbb R$ on a topological space $Y$, $A_\infty$ persistent homology 
builds a family of barcodes presenting a highly detailed description of some geometric and topological properties of $Y$. 
We prove here that under mild assumptions, these barcodes are stable: small perturbations in the function $f$ imply at most small perturbations in the barcodes. 
\end{abstract}


\section{Introduction}
\label{sec:intro}

Persistent homology (in the sense of
\cite{Carlsson-Zomorodian05,
Edelsbrunner-Letscher-Zomorodian02}) is a topological technique used to
extract global structural information
from datasets which may be high dimensional and contain noise.

About a decade ago, two results set the foundations of persistent homology as a robust mathematical theory.
First, the \emph{structural theorem} \cite[\S 3]{Carlsson-Zomorodian05} explained how the homology of a sequence of nested topological spaces can be split into simple pieces forming a barcode or a persistence diagram. Secondly, the \emph{stability theorem} \cite[Main Thm.]{Cohen_Steiner-Edelsbrunner-Harer07} showed that small perturbations in the input sequence can produce at most small perturbations in the corresponding barcode.

These two milestones justified the use of barcodes as a meaningful characteristic which is robust to noise. They also provided the formalism to show that in order to estimate the homology groups of a closed subspace $X$ of a metric space,
in theory it is enough to have a sufficiently good finite sample $P$ of $X$ \cite[Homology Inference Theorem]{Cohen_Steiner-Edelsbrunner-Harer07}.
To that end, one would only need to compute the barcode of the following sequence of nested spaces: for any given radius $r$, consider the union $P_r$ of the balls of radius $r$ centered at each point in $P$. Then, as $r$ grows, so does the union $P_r$.

Persistent homology has been successfully applied to fields such as medicine \cite{Belchi18_COPD, Adcock-Carlsson-Rubin14}, sensor networks coverage \cite{de_Silva-Ghrist07} and molecular modelling \cite{Gameiro-Hiraoka-Izumi-Kramar_Mischaikow-Nanda15, PirashviliEtAl_Chemistry_Soton18}, among many others. 
However, persistent homology computes information only at the level of homology groups. Intuitively, this means that persistent homology cares about the number of connected components, tunnels, voids and higher-dimensional holes of objects, and this information is not always enough.
For instance, work on signal processing \cite{Perea16_toroidal} and image texture representation \cite{Carlsson-Ishkhanov-de_Silva-Zomorodian08} shows that point clouds whose shapes are related to tori $\mathbb{T}$ and Klein bottles $\mathbb{K}$ arise naturally from data. With coefficients in the finite field of two elements $\mathbb{Z}_2 = \{0, 1\}$, homology groups do not distinguish $\mathbb{T}$ from $\mathbb{K}$, nor from a space as simple as a wedge of spheres $\mathbb{S}^1 \vee \mathbb{S}^2 \vee \mathbb{S}^1$, but the fundamental group does, and so does the cohomology ring.
It was then natural to enhance persistent homology with the discriminatory power of the fundamental group or the cohomology ring.
A persistence approach to the fundamental group can be found in 
\cite{Chazal-Lieutier05, Brendel-Dlotko-EllisEtAl15},
and the cup product is dealt with within the theory of $A_\infty$ persistent homology, or $A_\infty$-persistence, for short \cite{Belchi-Murillo15, Belchi17}. Beyond that, in order to use cohomology to detect that the Borromean rings are non-trivially linked, the cup product is not enough, and ternary operations like Massey products are needed. Information at this ternary and $n$-ary level in general is included as well in the computations of $A_\infty$ persistent homology.

Recent advances in generalizing the structure theorem \cite{chazal2016structure,Crawley-Boevey15,Lesnick15_ThyOfInterleaving} and in categorifying the stability theorem \cite{Bubenik_GenPersMod_15,de2018theory} allow one to 
prove that a given version of persistence (such as $A_\infty$ persistent homology) has a barcode decomposition and is stable, provided it is functorial.
The issue is that $A_\infty$ persistent homology is not functorial in general \cite[Thm. 3.1]{Belchi-Murillo15}. 
Therefore, a big challenge consists of finding a non-trivial context in which we can guarantee the functoriality of $A_\infty$-persistence. The main contribution of this work is the identification of one such context. Specifically, we introduce the category $\Top_n$ (Def. \ref{Def:Top_m}) and show that $A_\infty$ persistent homology is functorial within this category (Thm. \ref{thm:Functoriality}).
Additionally, we illustrate that this is the
largest category of its form (in a sense made explicit in Rmk. \ref{rmk:Top_n_is_the_largest_4_functoriality}) for which such functoriality should be expected.

A crucial part of this paper is therefore devoted to proving the functoriality of $A_\infty$ persistent homology (Thm.~\ref{thm:Functoriality}). From this, we then deduce that the barcodes from $A_\infty$ persistent homology are robust to small perturbations of the input (Cor.~\ref{cor:Stability_A_infty} and Cor.~\ref{cor:stability_Hausdorff}), and that one can extrapolate $A_\infty$ information of a metric space from a finite point-set approximation (Cor.~\ref{cor:ainfty-inference}).

This paper is organized as follows: 
In \S \ref{sec:persistenceAndFunctoriality}, we recall the basics of persistent homology and state the formal results we will use in \S \ref{sec:stability_of_Ainfty} to study the stability of $A_\infty$ persistent homology.
In \S \ref{sec:A_inftyCoalgebras}, we collect all definitions and properties we need to know about $A_\infty$-structures in order to understand the theory of $A_\infty$ persistent homology. All results in \S \ref{sec:stability_of_Ainfty} are stated in terms of the category $\Top_n$ we define in Def.~\ref{Def:Top_m}. 
The main theorem of the paper proves the functoriality of $A_\infty$ persistent homology (Thm.~\ref{thm:Functoriality}).
Rmk.~\ref{rmk:Top_n_is_the_largest_4_functoriality} and Ex. \ref{ex:Top_n_is_the_largest_4_functoriality} illustrates that $\Top_n$ is large enough in a particular sense.
As a first corollary of Thm.~\ref{thm:Functoriality}, we provide a new structure theorem
for $A_\infty$ persistent homology for a case left aside to date (Cor.~\ref{cor:Barcode_Ainfty}). 
To illustrate the higher discriminatory power of $A_\infty$ persistent homology over classical persistence, Ex. \ref{ex:barcodes_in_PH_and_Ainfty} exhibits two persistent spaces $X_*, Y_*$ with the same persistent homology barcodes but different $A_\infty$ barcodes. This toy example also shows how cup product persistence is part of $A_\infty$-persistence.
We finish \S \ref{sec:stability_of_Ainfty} with three important applications of the functoriality Thm.~\ref{thm:Functoriality} - namely, we show that $A_\infty$ persistent homology is stable with respect to perturbations in the input function (Cor.~\ref{cor:Stability_A_infty}) and perturbations in the input space (Cor.~\ref{cor:stability_Hausdorff}), and that we can recover $A_\infty$ information of a metric space from a finite point sample (Cor.~\ref{cor:ainfty-inference}).

Note that when we focus on the second operation $\Delta_2$ on an $A_\infty$-coalgebra (such as in Ex. \ref{ex:barcodes_in_PH_and_Ainfty}), or equivalently, on the cup product on an $A_\infty$-algebra on cohomology, then all results in this paper hold without the need to restrict to the category $\Top_n \subseteq \Top$ and instead, we can work directly with the category of topological spaces $\Top$. In particular, this paper proves the stability of the persistence of cup product with minimal restrictions.

\begin{notation*}
Throughout the text, we will work over a fixed field $\mathbb F$. We will usually omit the field from the notation. \emph{E.g., } we will denote by $H_*(X)$ the singular homology of $X$ with coefficients in $\mathbb F$.

We will present the results of this paper in terms of homology, but everything works as well for cohomology and for reduced (co)homology.
\end{notation*}


\section{Persistence and functoriality}
\label{sec:persistenceAndFunctoriality}

Let 
\begin{equation}
\label{eq:filtration_discrete}
\xymatrix{
K_0 \ar@{^{(}->}[r] & K_1 \ar@{^{(}->}[r] & \cdots \ar@{^{(}->}[r] & K_N
}
\end{equation}
be a finite sequence of nested topological spaces.
In the context of persistence, sequences like this arise as sublevel sets of functions of the form $f \colon M \longrightarrow \mathbb{R}$, for some metric space $M$; for instance, by specifying $$K_i \coloneqq f^{-1}( -\infty, i ].$$

To give a more concrete example, given a closed subspace $X$ of $M$, if one defines the distance function
$$ d^X \colon M \to \mathbb{R}, \quad y \mapsto d(y, X),$$
then the sequence given by $K_i \coloneqq (d^X)^{-1}( -\infty, i ]$
can be interpreted as a thickening of $X$.

Let us fix a particular homology degree of interest, $p \geq 0$ and assume that all these nested spaces have finite-dimensional homology groups, \emph{i.e.,} $\dim_\mathbb{F} H_p(K_i) < \infty$, for all $0 \leq i \leq N$.
We learned from \cite{Carlsson-Zomorodian05} that we can decompose the $p^\text{th}$ homology of the sequence (\ref{eq:filtration_discrete}) in simple pieces that can be represented in what is called a barcode or a persistence diagram. We now recall the formalism behind this in a higher level of generality which we will need later on.

\begin{notation*}
$\bR$ will denote the poset $(\bR,\leq)$ of real numbers.  
$\Vect$ will denote the category of $\mathbb F$-vector spaces and linear maps, and $\Top$ will denote the category of topological spaces and continuous maps.

We will use the notation $\mathcal P$ for any poset $(\mathcal P, \leq)$, i.e. any category whose objects are the elements of $\mathcal P$, and such that given two objects $x, y \in \mathcal P$, there is exactly one arrow $x \rightarrow y$ if $x \leq y$ and no arrow $x \rightarrow y$, otherwise.
\end{notation*}

\begin{definition}
Let $\CC$ be any category.
A \textbf{generalized persistence module} (valued in $\CC$)
is a functor of the form $\Ffunc:\bR\to\CC$.
When the category $\CC$ is understood by the context, we call $\Ffunc$ a generalized persistence module or \textbf{GPM} for short.
A morphisms between GPMs is a natural transformation between these functors.
In this way, the collection of all GPMs forms a functor category $\CC^{\bR}$ which we call a GPM-category.
The GPM-categories we focus on are $\Top^{\bR}$ and $\Vect^\bR$, whose objects are called \textbf{persistence spaces} and \textbf{persistence modules}, respectively.

A persistence module $\Vfunc \in \Vect^\bR$ is \textbf{pointwise finite dimensional (p.f.d) } if $\dim_\mathbb{F}\Vfunc(t) < \infty$ for every $t\in\bR$.
\end{definition}

\begin{definition}
\label{def:d_infty_for_Rvalued_continuous_maps}
For a pair of continuous maps
$f:X \longrightarrow \bR, g:{Y} \longrightarrow \bR$,
let us define the distance
\begin{equation}
\label{eq:l-infinity_distance}
d_{\infty}(f,g)=\inf_{\Phi}||f-g\circ\Phi||_{\infty}
\end{equation}
where $\Phi$ ranges over all homeomorphisms of the form $\Phi:X\to{Y}$. 
We set 
$d_{\infty}(f, g)=\infty$
if $X$ and ${Y}$ are not homeomorphic.
\end{definition}

The collection of all real-valued continuous functions forms a slice category $(\Top\downarrow\bR)$ which is equipped with the distance given in Def.~\ref{def:d_infty_for_Rvalued_continuous_maps}.
These functions are commonly used to construct persistence spaces via the sublevel-set filtration construction $\SS$.

\begin{definition}
The \textbf{sublevel-set filtration functor}
$\SS \colon (\Top\downarrow\bR) \longrightarrow \Top^{\bR}$
assigns to each continuous $f:X\to\bR$ the persistence space $\SS(f):\bR\to\Top$, $t\mapsto f^{-1}(-\infty,t]$.
\end{definition}

Fix an integer $p \geq 0$. 
Consider the singular homology functor $ H_p \colon \Top\to\Vect$, $X\mapsto H_{p}(X)$ that assigns to each space its $p^\text{th}$ homology group with coefficients in the field $\mathbb{F}$.

\begin{definition}
The post composition functor $ H_p\circ-:\Top^{\bR}\to\Vect^{\bR}$ assigns to each persistence space ${X}_{*} \colon \mathbb{R} \longrightarrow \Top,$ $t\mapsto {X}_{t}$, the persistence module $ H_p{X}_{*}\colon \mathbb{R} \longrightarrow \Vect$,  $t\mapsto H_p({X}_{t})$. $H_p{X}_{*}$ is called 
the \textbf{persistent $p^\text{th}$ homology of ${X}_{*}$.}
Analogously, the post composition functor $ H_p\SS\circ-:(\Top\downarrow\bR)\to\Vect^{\bR}$ assigns to each continuous function $f:X\to\bR$, the persistence module $ H_p\SS(f)\colon \mathbb{R} \longrightarrow \Vect$,  $t\mapsto H_p(f^{-1}(-\infty,t])$. $ H_p\SS(f)$ is called the \textbf{sublevel-set persistent $p^\text{th}$ homology of $f:X\to\bR$.}
\end{definition}

There is a natural way of splitting a persistence module into elementary pieces called interval persistence modules:

\begin{definition}
Given an interval $I\subseteq \mathbb R \cup \{+\infty\}$, the \textbf{interval persistence module} $C(I) \in \Vect^{\mathbb R}$ is given by
\[ 
\Cfunc(I)(t)=
\begin{cases} 
\mathbb{F} & \text{if }t\in I \\
0 & \text{otherwise.} 
\end{cases}
\]
\[ 
\Cfunc(I)[s\leq t]=
\begin{cases} 
Id_{\mathbb{F}} & \text{if }s,t\in I \\
0 & \text{otherwise.} 
\end{cases}
\]
\end{definition}

\begin{theorem} 
\textbf{(Structure theorem of persistence modules \cite{Crawley-Boevey15})}
\label{thm:Barcode_decomposition}
Every p.f.d. persistence module $\Vfunc \in \Vect^{\mathbb R}$ decomposes uniquely (up to isomorphism)  
into interval persistence modules $C(I)$,
$$\Vfunc\cong\bigoplus_{I\in B(\Vfunc)}\Cfunc(I),$$
where $B(\Vfunc)$ is a multiset (\emph{i.e.,} a set of objects with multiplicities) of intervals of the form $[a,b)$ or $(-\infty, b)$ for some $a\in \mathbb R$, $b\in  \mathbb{R}\cup\{+\infty\}$. This $B(\Vfunc)$ is called the \textbf{barcode} of $\Vfunc$.
\end{theorem}

In particular, the persistence module $\Vfunc \coloneqq H_p{X}_{*} \in \Vect^{\mathbb R}$ is uniquely determined by a barcode, as long as 
$\dim_\mathbb{F} H_p({X}_{t}) < \infty$, for all $t\in\mathbb{R}$.
In comparison, the original structure theorem \cite[\S 3]{Carlsson-Zomorodian05} required this persistence module to be indexed by the integers, $ H_p{X}_{*} \in \Vect^{\mathbb Z}$.
Additionally, \cite{Carlsson-Zomorodian05} also required there to be finitely many points $r \in \mathbb{R}$ such that, for every sufficiently small $\epsilon >0$, the linear map $H_p{X}_{r-\epsilon} \longrightarrow H_p{X}_{r + \epsilon}$ is not an isomorphism.

Each barcode has an associated persistence diagram, which is a multiset of points in the extended plane 
$ \left( \mathbb R \cup \{-\infty\} \right)
\times \left( \mathbb R \cup \{+\infty\} \right)$.

\begin{definition}
Given a barcode $B(\Vfunc)$, its corresponding \textbf{persistence diagram} $Dgm(\Vfunc)$ consists of the following points:
\begin{itemize}
\item For each interval $[a,b) \in B(\Vfunc)$ with multiplicity $k$, the point of coordinates $(a, b) \in \mathbb R \times \left( \mathbb R \cup \{+\infty\} \right)$ is included in $Dgm(\Vfunc)$ with multiplicity $k$.
\item For each interval $(-\infty,b) \in B(\Vfunc)$ with multiplicity $k$, the point of coordinates $(-\infty, b) \in \{-\infty\}
\times \left( \mathbb R \cup \{+\infty\} \right)$ is included in $Dgm(\Vfunc)$ with multiplicity $k$.
\item For each $x \in \mathbb R$, the diagonal point $(x, x)\in \mathbb R^2$ is included in $Dgm(\Vfunc)$ with infinite multiplicity.
\end{itemize}
\end{definition}

Many stability theorems are presented in terms of a pseudo-distance between barcodes or persistence diagrams called the bottleneck distance.
In order to give its definition,
we use the following notation. For any $(p_1, p_2), (q_1, q_2) \in \left( \mathbb R \cup \{-\infty\} \right) \times \left( \mathbb R \cup \{+\infty\} \right)$,
$$ || (p_1, p_2) - (q_1, q_2) ||_\infty \coloneqq \max \{ |p_1 - q_1|, |p_2 - q_2|\},$$ 
where $|\cdot|$ denotes the absolute value with the convention that:
\begin{itemize}
\item If $p_2=q_2=+\infty$, then $|p_2 - q_2| \coloneqq 0$.
\item If $p_1=q_1=-\infty$, then $|p_1 - q_1| \coloneqq 0$.
\item If one, and only one, of $p_2, q_2$ is $+\infty$, then $|p_2 - q_2| \coloneqq +\infty$.
\item If one, and only one, of $p_1, q_1$ is $-\infty$, then $|p_1 - q_1| \coloneqq +\infty$.
\end{itemize}

We think of each point of multiplicity $k$ in a persistence diagram as $k$ different points, and since all diagonal points are added with infinite multiplicity, there are always infinitely many bijections between any given pair of persistence diagrams.

\begin{definition}
\label{def:bottleneck_distance}
The \textbf{bottleneck distance} $d_B$ between two barcodes $B(\Vfunc)$ and $B(\Wfunc)$ (or between their associated persistence diagrams $Dgm(\Vfunc)$ and $Dgm(\Wfunc)$) is defined as
$$ d_B( B(\Vfunc), B(\Wfunc)) = d_B (Dgm(\Vfunc), Dgm(\Wfunc)) \coloneqq \inf_\gamma \sup_x || x - \gamma(x) ||_\infty, $$
where $x$ runs over all points in $Dgm(\Vfunc)$ and $\gamma$ runs over all bijections $\gamma \colon Dgm(\Vfunc) \longrightarrow Dgm(\Wfunc)$.
\end{definition}
See \cite[\S 3.1]{Lesnick15_ThyOfInterleaving}
for a definition of the bottleneck distance in terms of matchings and see \cite{KerberEtAl16} to learn how to efficiently compute the bottleneck distance.

\begin{example}
\label{ex:Bottleneck_distance}
Let $\Vfunc$ and $\Wfunc$ be the persistence modules given by the following barcodes:
$B(\Vfunc) = \{ [1, 6), [9, 10) \}, B(\Wfunc) = \{[2, 6.8) \}.$
The bottleneck distance $d_B(B(\Vfunc), B(\Wfunc))$ can be computed as $\max \{x, y, z\}$, where $x, y, z$ are those distances shown in Fig.~\ref{fig:Bottleneck_distance_example}.
Hence, $d_B(B(\Vfunc), B(\Wfunc)) = \max \{1, 0.8, 0.5\} = 1.$
\end{example}

\begin{figure}[h!]
    \centering
	\includegraphics[width=0.9\textwidth]{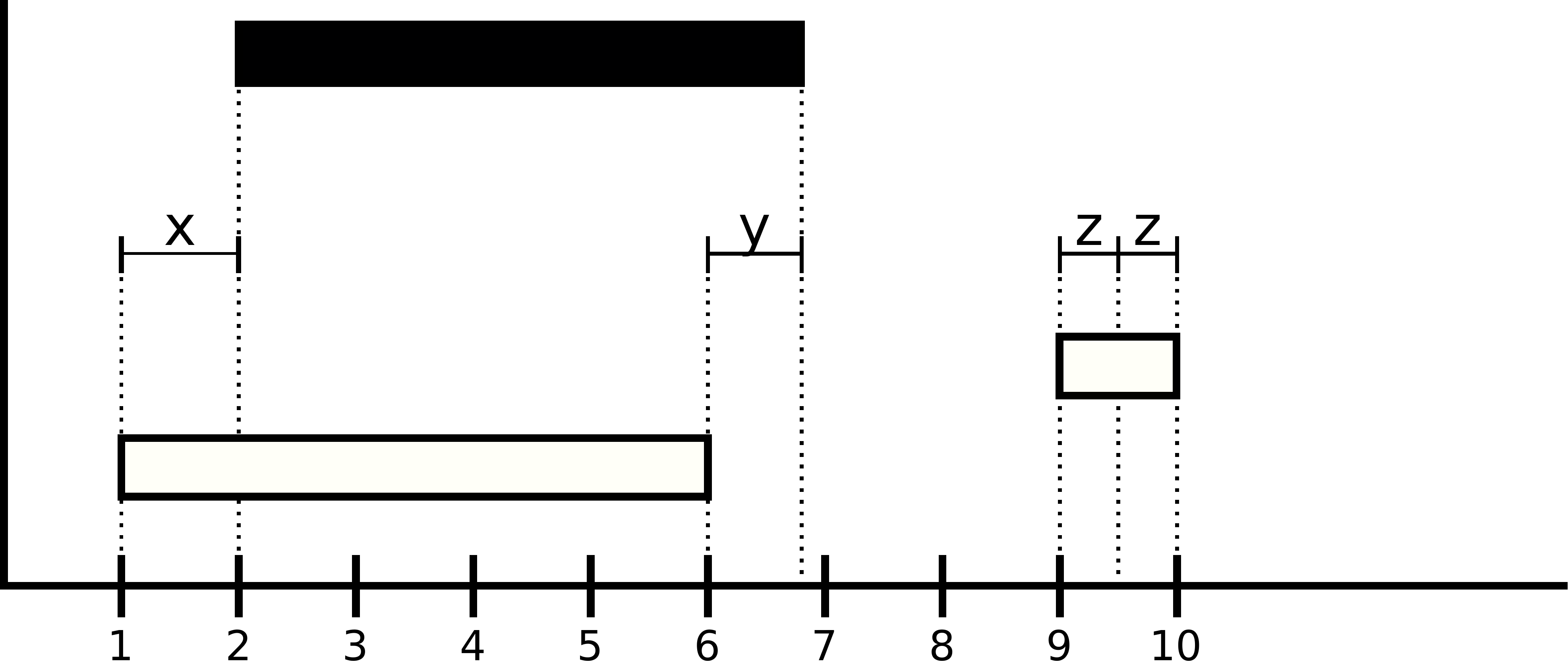}
	\caption{Illustration of the bottleneck distance.
	The two intervals in the barcode
	$B(\Vfunc) = \{ [1, 6), [9, 10) \}$ are shown in white and the interval of the barcode $B(\Wfunc) = \{[2, 6.8) \}$ is shown in black. Therefore, 
    $d_B(B(\Vfunc), B(\Wfunc)) = \max \{x, y, z\} = \max \{1, 0.8, 0.5\} = 1.$
    Indeed, the bottleneck distance $d_B(B(\Vfunc), B(\Wfunc))$ can be interpreted as the minimum amount one has to enlarge or shrink the ends of the intervals in $B(\Vfunc)$ in order to obtain the intervals in $B(\Wfunc)$.
	Just as we assumed every persistence diagram to contain each diagonal point with infinite multiplicity, here we can think of each barcode as having each interval of the form $[a, a) = \emptyset$, for $a\in\mathbb R$, with infinite multiplicity.
	This allows the trick of shrinking both ends of an interval by half of its length to make it disappear and the trick of enlarging both ends of an empty interval $[a,a)$ to form an interval $[a-\epsilon, a+\epsilon)$.}
	\label{fig:Bottleneck_distance_example}
\end{figure}

In order to state the classical stability theorem for barcodes, we need one last concept.
    
\begin{definition}
\label{def:homological_critical_value}
Given a continuous function $f \colon {X} \longrightarrow \bR$, its \textbf{homological critical values} of degree $p$ are those points $t \in \mathbb{R}$ such that, for every sufficiently small $\epsilon >0$, the linear map $H_p{X}_{t-\epsilon} \longrightarrow H_p{X}_{t + \epsilon}$ is not an isomorphism.
\end{definition}
    
\begin{theorem}
\label{thm:stability_classical}
\textbf{(Stability of persistent homology 
\cite[Main Thm.]{Cohen_Steiner-Edelsbrunner-Harer07})}
Let $p \geq 0$ be an integer, let ${X}$ be a triangulable space and let
$f, g \colon {X} \longrightarrow \bR$ be continuous functions with finitely many homological critical values of degree $p$. 
If the persistence modules $ H_p\SS(f), H_p\SS(g) \colon \mathbb{R} \longrightarrow \Vect$ are p.f.d., then the bottleneck distance between the barcodes of $f$ and $g$ is bounded above by the supremum distance between the functions:
$$d_B(B(H_p\SS(f)), B(H_p\SS(g))) \leq ||f - g||_\infty.$$
\end{theorem}

Thm.~\ref{thm:stability_classical} guarantees that slight changes in the input function can only result in slight changes in the corresponding barcodes.
Although this is considered the most classical stability result, it must be acknowledged that earlier work by the Italian team of M. d'Amico et al. obtained similar results for degree 0 homology \cite{d'Amico-Frosini-Landi03}.

We now recall the rest of the tools we need on the categorification of stability, which will help us prove that $A_\infty$ persistent homology is stable (see \S \ref{sec:stability_of_Ainfty}).
For that purpose, we will need to make some extra assumptions, but at the same time, we will also drop two of the assumptions in Thm.~\ref{thm:stability_classical} -- namely, the triangulability of ${X}$, and the finiteness condition on the number of homological critical points.

We first loot at interleavings between functors, which provide us with a tool to compare persistence modules and to compare persistence spaces as well.

\begin{definition}
	\cite{ChazalEtAl09_proximity}
	Two GPMs $\Ffunc,\Gfunc:\bR\to\CC$ are \textbf{$\epsilon$-interleaved},
	for $\epsilon\geq0$,
	if there exists a pair of natural transformations 
	$\varphi_t:\Ffunc(t)\to\Gfunc(t+\epsilon)$, $t\in\bR$ 
	and 
	$\psi_t:\Gfunc(t)\to\Ffunc(t+\epsilon)$, $t\in\bR$  
	such that the following diagrams commute
	\begin{equation*}
	\begin{tikzcd}
	\Ffunc(t)
	\arrow[dd, "{\Ffunc[t\leq t+2\epsilon]}"']
	\arrow[dr,  "{\varphi_t}"']
	\&
	\&
	\&
	\Gfunc(t)
	\arrow[dl, "{\psi_t}"']
	\arrow[dd, swap, "{\Gfunc[t\leq t+2\epsilon]}"']
	\\
	\&
	\Gfunc(t+\epsilon)
	\arrow[dl, crossing over,  "{\psi_{t+\epsilon}}"]
	\&
	\Ffunc(t+\epsilon)
	\arrow[dr, crossing over, "{\varphi_{t+\epsilon}}"]
	\\
	\Ffunc(t+2\epsilon)\&\&\&
	\Gfunc(t+2\epsilon).
	\end{tikzcd}
	\end{equation*}
	The \textbf{interleaving distance} between $\Ffunc$ and $\Gfunc$ is then defined as 
	\begin{equation*}
	d_I(\Ffunc,\Gfunc) \coloneqq \inf\{\epsilon\geq0 \mid \Ffunc,\Gfunc\text{ are }\epsilon\text{-interleaved}\}.
	\end{equation*}
	If $\Ffunc$ and $\Gfunc$ are not $\epsilon$-interleaved for any $\epsilon$, we set $d_I(\Ffunc,\Gfunc) = \infty$.
\end{definition}

In a way, the interleaving distance $d_I(\Ffunc,\Gfunc)$ measures how far the GPMs $\Ffunc$ and $\Gfunc$ are from being isomorphic. For instance, if two GPMs $\Ffunc$ and $\Gfunc$ are 0-interleaved, it means that they are isomorphic.

\begin{example}
\label{ex:Interleaving_distance}
Let $\Vfunc$ and $\Wfunc$ be the persistence modules in Ex.~\ref{ex:Bottleneck_distance} given by the following barcodes:
$B(\Vfunc) = \{ [1, 6), [9, 10) \}, B(\Wfunc) = \{[2, 6.8) \}.$
The interleaving distance $d_I(\Vfunc, \Wfunc)$ cannot be smaller than 1, since
for every $\epsilon < 1$, the diagram on the left is not commutative (the values of this diagram are made explicit in the diagram next to it):
\begin{equation*}
	\begin{tikzcd}
	\Vfunc(1)
	\arrow[dd, "{\Vfunc[1\leq 1+2\epsilon]}"']
	\arrow[dr,  "{\varphi_1}"']
    \& \& \&
    \mathbb F
	\arrow[dd, "{1_{\mathbb F}}"']
	\arrow[dr,  "{0}"']
	\\
	\&
	\Wfunc(1+\epsilon)
	\arrow[dl, crossing over,  "{\psi_{1+\epsilon}}"]
    \& \& \&
	0
	\arrow[dl, crossing over,  "{0}"]
	\\
	\Vfunc(1+2\epsilon)
    \& \& \&
    \mathbb F
	\end{tikzcd}
\end{equation*}
Notice, though, that for $\epsilon = 1$, the following diagram is commutative:
\begin{equation*}
	\begin{tikzcd}
	\Vfunc(1)
	\arrow[dd, "{\Vfunc[1\leq 3]}"']
	\arrow[dr,  "{\varphi_1}"']
    \& \& \&
    \mathbb F
	\arrow[dd, "{1_{\mathbb F}}"']
	\arrow[dr,  "{1_{\mathbb F}}"']
	\\
	\&
	\Wfunc(2)
	\arrow[dl, crossing over,  "{\psi_{2}}"]
    \& \& \&
	\mathbb F
	\arrow[dl, crossing over,  "{1_{\mathbb F}}"]
	\\
	\Vfunc(3)
    \& \& \&
    \mathbb F
	\end{tikzcd}
\end{equation*}
It is straightforward to check that, indeed, $d_I(\Vfunc, \Wfunc) = 1.$	
\end{example}

\begin{theorem}\cite[Thm.~3.21]{Bubenik_GenPersMod_15}
The interleaving distance $d_I$ on a GPM-category $\CC^{\bR}$ is an extended pseudometric on the class of generalized persistence modules. 
\end{theorem}

It is no coincidence that the distances computed in Ex.~\ref{ex:Bottleneck_distance} and Ex.~\ref{ex:Interleaving_distance} are the same. Indeed, there is a deep relation between p.f.d. persistence modules and their barcodes known as the isometry theorem.

\begin{theorem}\textbf{(Isometry Theorem
\cite{ChazalEtAl09_proximity, Lesnick15_ThyOfInterleaving})}
\label{thm:Isometry_theorem}
	The interleaving distance of a pair of p.f.d. persistence modules $\Vfunc,\Wfunc \in {Vect}^\mathbb{R}$ is equal to the bottleneck distance of their associated barcodes, i.e.
	$$d_I(\Vfunc,\Wfunc)=d_B(B(\Vfunc),B(\Wfunc)).$$
\end{theorem}

F. Chazal et al.~\cite{ChazalEtAl09_proximity} proved that
$d_I(\Vfunc,\Wfunc) \geq d_B(B(\Vfunc),B(\Wfunc)),$
and more recently, M. Lesnick \cite{Lesnick15_ThyOfInterleaving} proved the converse inequality.

As Thm.~\ref{thm:stability_classical} illustrates, the study of the bottleneck distance $d_B$ and its generalizations is of central importance to TDA. Thm.~\ref{thm:Isometry_theorem} allows us to prove useful results concerning $d_B$ by understanding the interleaving distance.

We next state some results we will use on the interleaving distance. We start by recalling that some functors can be viewed as non-expansive maps:

\begin{theorem}
\cite[Cor.~of Thm.~3.16]{Bubenik_GenPersMod_15}
\label{thm:Functors_are_1Lipschitz_Real_case}
	Let $\Hfunc : \CC \to \DD$ be a functor between arbitrary categories $\CC, \DD$.
	Then the functor $\Hfunc\circ- : \CC^{\bR}\to\DD^{\bR}$, $\Ffunc\mapsto\Hfunc\Ffunc$, defined by post-composing with $\Hfunc$, is 1-Lipschitz.
	That is, for any two GPMs $\Ffunc, \Gfunc : \bR \to\CC$, we have 
    $$d_I(\Hfunc\Ffunc, \Hfunc\Gfunc)\leq d_I(\Ffunc, \Gfunc).$$
\end{theorem}

By Thm.~\ref{thm:Functors_are_1Lipschitz_Real_case}, the persistent $p^\text{th}$ homology $ H_p\circ-$ forms a 1-Lipschitz map with respect to the interleaving distance.
That is, for any pair of persistence spaces $\X_{*}:\bR\to\Top$, $t\mapsto{X}_t$ and ${Y}_{*}:\bR\to\Top$, $t\mapsto{Y}_t$ 
we have 
$$d_I( H_p{X}_{*}, H_p{Y}_{*})\leq d_I({X}_{*},{Y}_{*}).$$

\begin{theorem}\cite{ChazalEtAl09_proximity}
\textbf{(The sublevel-set functor is non-expansive)}
\label{thm:SublevelSet_1Lipschitz}
The sublevel-set filtration map $\SS \colon (\Top\downarrow\bR) \longrightarrow \Top^{\bR}$ is 1-Lipschitz, \emph{i.e.,} if $f \colon X \longrightarrow \mathbb R$ and $g \colon Y \longrightarrow \mathbb R$ are continuous, then
$$d_I(\SS(f),\SS(g))\leq d_{\infty}(f, g).$$
\end{theorem}
	
\section{Transferred $A_\infty$-coalgebras on $  H_*(X)$}
\label{sec:A_inftyCoalgebras}

One can extract topological information of a space from the structure of its homology. To do this beyond the Betti numbers, one can study the relationship between homology classes.
For instance, a torus $\mathbb T = \bS^1 \times \bS^1$ and a wedge of spheres $\bS^1 \vee \bS^2 \vee \bS^1$ have the same Betti numbers but their homology classes are related in very different ways.
On the one hand, $\mathbb T$ is a surface of revolution whose generatrix curve is a circumference which creates non-trivial homology. This hints some relation between entities of dimension 1 (the curve) and 2 (the surface), which, in terms of cohomology, can be explained as follows: choosing an appropriate basis of the cohomology $H^*(\bT)$, the generator of $H^2(\bT)$ is the cup product of the two generators of $H^1(\bT)$. In contrast, the generatrix of a sphere $\bS^2$ does not produce non-trivial homology in $\bS^1 \vee \bS^2 \vee \bS^1$. In terms of cohomology, the generator of $H^2(\bS^1 \vee \bS^2 \vee \bS^1)$ is nobody's cup product, and the cup product of the 2 generators of $H^1(\bS^1 \vee \bS^2 \vee \bS^1)$ is 0, regardless of the chosen basis of $H^*(\bS^1 \vee \bS^2 \vee \bS^1)$.
$A_\infty$-structures contain all the information provided by the cup product and much more. For instance, there are links which cannot be distinguished by the cup product alone, but which can be distinguished using $A_\infty$-structures
\cite{Massey68/98}, \cite[\S 3]{Belchi17}. Later in this section, we will mention more examples of spaces for which the $A_\infty$-structure provide much more detailed topological information.

In this section, we list some basic notions and notation we will need to understand the meaning of $A_\infty$ persistent homology and its stability. We will restrict ourselves to $A_\infty$-coalgebra structures on graded vector spaces only, although they can be defined in more general contexts.

\begin{definition}\label{DEF:A_infty-coalgebra}
An \textbf{${\bm A_\infty}$-coalgebra structure} 
$\{\De_n\}_{n \geq 1}$ on a graded vector space $C$
is a family of maps
$$\Delta_n\colon  C\longrightarrow C^{\otimes n}$$
of degree $n-2$
such that, for all $n \geq 1$,
the following  \emph{Stasheff identity} holds:
$$
\text{SI}(n)\colon 
 \sum_{i=1}^n \sum_{j=0}^{n-i} (-1)^{i+j+ij}
		\left( 1^{\otimes n-i-j} \otimes \De_i \otimes 1^{\otimes j} \right)
		\,\De_{n-i+1} = 0. $$	
\end{definition}

If $(C, \{\De_n\}_{n \geq 1})$
is an $A_\infty$-coalgebra,
the identity $\text{SI}(1)$ states that 
$\De_1$ is a differential on $C$ and 
$\text{SI}(3)$ states that 
the comultiplication
$\De_2$ is coassociative up to the chain homotopy $\De_3$.
Any differential graded coalgebra $(C,\dlie,\De)$
can be viewed as an $A_\infty$-coalgebra 
$(C, \{\De_n\}_{n \geq 1})$ by setting
$ \Delta_1 = \dlie, \Delta_2 = \De,$ and $\Delta_n =0$ for all $n > 2.$
%
An $A_\infty$-coalgebra
$(C, \{\De_n\}_{n \geq 1})$
is called \textbf{minimal}
if 
$\De_1=0.$

\begin{definition}\label{DEF:Morph-A_coalg}
A \textbf{morphism of $\bm{A_\infty}$-coalgebras} 
$$ f\colon (C,\{\De_n\}_{n\ge 1})\to (C',\{\De'_n\}_{n\ge 1})$$
is a family of linear maps
 $$
 f_{(m)}\colon C\longrightarrow {C'}^{\otimes m},\qquad m\ge 1,
 $$
of degree $m-1$, such that
for each $i\ge 1$,
the following identity holds:
$$
\text{MI}(i)\colon \sum_{\substack{
p+q+k=i\\ q \geq 1, p,k \geq 0
}}
(1^{\otimes p}\otimes\Delta'_q\otimes 1^{\otimes k})f_{(p+k+1)}=
\sum_{\substack{
k_1+\dots+k_\ell=i\\ l, k_j \geq 1
}}
(f_{(k_1)}\otimes\dots\otimes f_{(k_\ell)})\Delta_\ell.
$$

We say that the
morphism of $A_\infty$-coalgebras $f$
is:
\begin{itemize}
\item
an \textbf{isomorphism}
if $f_{(1)}$ is an isomorphism (of vector spaces),
\item
a \textbf{quasi-isomorphism}
if $f_{(1)}$ induces an isomorphism (of vector spaces)
in homology.
\end{itemize}
\end{definition}

There are some trivial $A_\infty$-coalgebra structures one can always endow a graded vector space $C$ with, such as the one given by $\Delta_n = 0$ for all $n$.
We hence only consider \emph{transferred} $A_\infty$-coalgebras.

\begin{definition}\label{DEF:good-A_coalg}
We will say that an $A_\infty$-coalgebra
$\left( H_*(X), \{ \De_n \}_n \right)$ 
on the homology of a space $X$
is a \textbf{transferred
$\bm{A_\infty}$-coalgebra} (induced by $X$) if it is minimal and
quasi-isomorphic
to the 
$A_\infty$-coalgebra
$$\left( C_*(X), \{\dlie, \De, 0, 0, \ldots \} \right),$$
where $\left( C_*(X), \dlie \right)$
denotes the singular chain complex of $X$
and $\De$ denotes the Alexander-Whitney diagonal.
We will drop the \emph{`induced by $X$'} from the notation
when no confusion is possible.
\end{definition}

The dual of this notion consists of transferred 
$A_\infty$-algebras $\{\mu_n\}_{n}$ on the cohomology of $X$,
$H^*(X)$, where $\mu_2$ coincides with the cup product.
Hence, transferred $A_\infty$-structures encode all the information in the homology
groups of $X$ and in its cohomology algebra as well, but there is more.
For instance, T.~Kadeishvili proved
that under mild conditions on a topological space $X$,
any transferred $A_\infty$-algebra on its cohomology
determines the cohomology of its loop space \cite[Prop.~2]{Kadeishvili80}, $H^*(\Omega X)$,
whereas the cohomology ring of $X$ alone
does not; in \cite[Thm.~1.3]{Belchi-Murillo15}, one can find a way to build pairs of spaces with isomorphic homology groups and isomorphic cohomology algebras 
but non-isomorphic transferred $A_\infty$-coalgebras;
in \cite[\S 3]{Belchi17}, \cite{Massey68/98}, one can find examples of links which are told apart by transfered $A_\infty$-structures, which cannot be told apart by using the cup product alone. This can be done thanks to the relation between Massey products and $A_\infty$-structures \cite{Buijs-Moreno-Murillo18_Massey}.

An immediate consequence of Definition \ref{DEF:good-A_coalg} is that
all transferred
$A_\infty$-coalgebras on $H_*(X)$
induced by $X$
are isomorphic.
The following is a folklore result, of which one can find a proof in
\cite[Cor.~3.3]{Belchi17}.

\begin{theorem}
\textbf{(Homotopy invariance of $\bm{\dim \Ker {\De_m}_{| H_p(X)}}$)}
\label{COR:min-invariant}
Let $ \{\De_n\}_n $ be a transferred 
$A_\infty$-coalgebra structure on the homology of a space $X$, and let us set
$$
k(X) \coloneqq
\begin{cases}
 \min \{ n \,|\, \Delta_n \neq 0 \} , & \text{if } \{ n \,|\, \Delta_n \neq 0 \} \neq \emptyset \\
+\infty, & \text{otherwise.}
\end{cases}
$$
Then, the number $k(X)\in\mathbb{Z}\cup \{+\infty\}$ and the integers
$ \dim \Ker {\De_m}_{|  H_p(X)} $ (for integers $m \leq k(X)$ and $p\geq 0$) are
independent of the choice of transferred $A_\infty$-coalgebra structure on $H_*(X)$. 
Moreover, since homotopy equivalent spaces induce isomorphic transferred $A_\infty$-coalgebras, $k(X)$ and every such $ \dim \Ker {\De_m}_{|  H_p(X)} $ are invariants of the homotopy type of $X$.
\end{theorem}

We finally recall a classical results used for computing and working with transferred $A_\infty$-structures which we will use in the proof of the functoriality of $A_\infty$ persistent homology (Thm. \ref{thm:Functoriality}).

\begin{theorem}
\cite{Kadeishvili80,
Loday-Vallette12}
\textbf{(Homotopy Transfer Theorem)} 
\label{thm:HTT}
Let the following be a diagram of chain complexes,
\begin{equation}\label{eq:transfer_diagram}
\xymatrix{ \ar@(ul,dl)@<-5.5ex>[]_\phi  & (M,d) \ar@<0.75ex>[r]^-\pi  & (N,d) \ar@<0.75ex>[l]^-\iota }
\end{equation}
where $(N, d)$ is a chain complex, $(M,d)$ is a differential graded coalgebra with comultiplication $\Delta$, and the degree 0 chain maps $\pi$ and $\iota$ and the degree 1 chain homotopy $\phi$ make the following hold:
$\pi \iota={\rm id}_N$, $\pi  \phi=\phi\iota=\phi^2=0$ and $\phi$ is a chain homotopy between ${\rm id}_M$ and $\iota \pi $, \emph{i.e.,} $\phi d+d\phi=\iota \pi -{\rm id}_M$. 
Then, there is an explicit minimal $A_\infty$-coalgebra structure $\{\De_n\}_n$ on $N$ with $\Delta_2=\pi^{\otimes 2} \Delta \iota$ and there are morphisms of $A_\infty$-coalgebras
$$
\xymatrix{M \ar@<0.75ex>[r]^-\Pi& N\ar@<0.75ex>[l]^-I },
$$
such that $\Pi_{(1)}=\pi$ and $I_{(1)}=\iota$.
\end{theorem}

Building a diagram of the form (\ref{eq:transfer_diagram}) with $M\cong C_*(X)$ and $N\cong H_*(X)$ amounts to building a transferred $A_\infty$-coalgebra structure on $H_*(X)$.

\section{$A_\infty$ persistent homology: topological estimation and stability}
\label{sec:stability_of_Ainfty}

The barcode of $H_p({X}_*)$ from classical persistent homology recovers information only at the level of homology groups. In contrast, the barcodes in $A_\infty$ persistent homology (which we recall in this section) consist of partial information from $A_\infty$-(co)algebras in (co)homology, therefore enhancing persistent homology with a greater discriminatory power.
In a sense, $A_\infty$ persistent homology could be viewed as a way to relate the different barcodes from classical persistent homology.

It would therefore be desirable to have stability results for $A_\infty$ persistent homology similar to those seen in \S \ref{sec:persistenceAndFunctoriality} for classical persistence.
The issue is that $A_\infty$ persistent homology is not functorial in general. Indeed, \cite[Thm. 3.1]{Belchi-Murillo15} illustrates that given 
transferred $A_\infty$-coalgebra structures 
$\left(H_*(X), \{\Delta_n^X\}_n \right)$,
$\left(H_*(Y), \{\Delta_n^Y\}_n \right)$
and a continuous map $f \colon X \longrightarrow Y,$
the inclusion
\begin{equation}
\label{eq:fKer_inKer}
  H_p(f)(\Ker {\Delta_n^X}_{| H_p(X)}) \subseteq \Ker {\Delta_n^Y}_{| H_p(Y)}
\end{equation}
does not need to hold, where $H_p(f)$ denotes the map induced in $p^{\text{th}}$ homology by $f$.
Here we tackle the challenge of finding a non-trivial context which guarantees the functoriality of $A_\infty$-persistence. In this section, 
we introduce the category $\Top_n$ (Def. \ref{Def:Top_m}) and show that $A_\infty$ persistent homology is functorial within this category (Thm. \ref{thm:Functoriality}).
Additionally, we illustrate that this is the
largest category of its form (in a sense made explicit in Rmk. \ref{rmk:Top_n_is_the_largest_4_functoriality}) for which such functoriality should be expected.
From Thm. \ref{thm:Functoriality}, stability results for $A_\infty$ persistent homology (Cor.~\ref{cor:Stability_A_infty}, \ref{cor:stability_Hausdorff} and \ref{cor:ainfty-inference}) will follow.

\begin{definition}
\label{Def:Top_m}
Let $n\in\mathbb{Z}\cup\{+\infty\}$, $n>1$, and let $\Top_n$ denote the category whose objects are topological spaces $X$ such that $\Delta_m = 0$, for all $m < n $,
where $\{\Delta_m\}_{m\geq1}$ denotes any transferred
$A_\infty$-coalgebra structure on $  H_*(X)$,
and where the morphisms are continuous maps.
\end{definition}

It follows from Def. \ref{Def:Top_m} that for every integer $n>1$, $\Top_\infty \subseteq \Top_{n+1} \subseteq \Top_n \subseteq \Top_2 = \Top$ are full subcategories. The higher $n\in\mathbb{Z}\cup\{+\infty\}$ is, the closer the objects in $\Top_n$ are to having the whole $A_\infty$-coalgebra structure on their homology fully determined by their cohomology ring (a notion related to formality in the context of rational homotopy theory).

Given a topological space $X$,
all the transferred $A_\infty$-coalgebra structures $\{\Delta_m\}_m$ on $  H_*(X)$ are isomorphic. 
The Axiom of Choice guarantees that we can fix a choice of transferred $A_\infty$-coalgebra structure $\{\Delta_m\}_m$ on $  H_*(X)$, for every $X \in \Top$. 
For a concrete example in a restricted scenario, let us consider filtered CW complexes, which are CW complexes $\langle c_0, \dots, c_m \rangle$ whose cells are ordered $c_0, \dots, c_m$ so that, for all $0 \leq i < m$, $\langle c_0, \ldots, c_i \rangle$ forms a subcomplex of $\langle c_0, \ldots, c_{i+1} \rangle$.
H. Molina-Abril and P. Real \cite[Alg.~1]{Molina_Abril-Real09} used discrete vector fields to create a deterministic algorithm which computes a transferred $A_\infty$-coalgebra structure on $H_*(X)$ for every filtered CW complex $X$.
Once a method to build $A_\infty$-coalgebra structures has been fixed, we can define $\kappa_{n,p}$:

\begin{definition}
\label{Def:Functor_n_p}
Let $\kappa_{n,p}: \Top_n \longrightarrow \Vect$ be the following assignment:
for every object $X$ in $\Top_n$, pick a particular transferred $A_\infty$-coalgebra structure $\{\Delta_m\}_m$ on $  H_*(X)$ and define
$$\kappa_{n,p}(X) \coloneqq \Ker {\Delta_n}_{|  H_p(X)} \subseteq   H_p(X). $$
For every morphism $f:X \longrightarrow Y$ in $\Top_n$, define
$$ \kappa_{n,p}(f): \kappa_{n,p}(X) \longrightarrow \kappa_{n,p}(Y)$$
as the map 
$$  H_p(f):   H_p(X) \longrightarrow   H_p(Y)$$ induced by $f$ in degree-$p$  homology, restricted to $\kappa_{n,p}(X)$. Being pedantic with the notation, this would be
$$\kappa_{n,p}(f) = {  H_p(f)}_{| \Ker {\Delta_n}_{|  H_p(X)}}. $$
\end{definition}

\begin{theorem}
\label{thm:Functoriality}
\emph{\textbf{(Functoriality of $\bm{\kappa_{n,p}}$)}}
For each pair of integers $n\geq 2$, $p \geq 0$,
the category $\Top_n$ is well defined,
the assignment $$\kappa_{n,p}: \Top_n \longrightarrow \Vect$$
in Def. \ref{Def:Functor_n_p} defines a functor and
the integer 
$\dim \kappa_{n,p}(X)$
does not depend on the choice of $A_\infty$-coalgebra made in Def. \ref{Def:Functor_n_p}.
\end{theorem}

\begin{proof}
By definition, the objects in $\Top_n$ are the topological spaces $X$ such that 
$$ \min \{ m \,|\, \Delta_m \neq 0 \} \geq n, $$
where $\{\Delta_m\}_m$ denotes any transferred
$A_\infty$-coalgebra structure on $H_*(X)$.
Hence, Thm.~\ref{COR:min-invariant} guarantees that $\Top_n$ is well defined for all $n\in\mathbb{Z}\cup\{+\infty\}$ such that $n>1$.
In particular, the property of $X$ being in $\Top_n$ does not depend on the choice of $A_\infty$-coalgebra on its homology. Rather, it only depends on the homotopy type of $X$.
Thm.~\ref{COR:min-invariant} also guarantees that if $X\in\Top_n$ for some integer $n>1$, then the integer 
$\dim \kappa_{n,p}(X) = \dim \Ker {\Delta_n}_{|H_p(X)}$
does not depend on the choice of $A_\infty$-coalgebra and it is indeed a homotopy invariant of $X$, for all $p \geq 0$.

Shifting the focus now to $\kappa_{n,p}$,
the Axiom of Choice makes the assignment $ X \mapsto \kappa_{n,p}(X)$ well defined. 
To prove that the assignment 
$$ f \mapsto \left( \kappa_{n,p}(f) \colon \kappa_{n,p}(X) \rightarrow \kappa_{n,p}(Y) \right)$$
is well defined too, is equivalent to showing that
$$ \Image \kappa_{n,p}(f) \subseteq \kappa_{n,p}(Y)$$
holds. For this,
it suffices to show that the inclusion in (\ref{eq:fKer_inKer}),
$
  H_p(f)(\Ker {\Delta_n^X}_{|    H_p(X)}) \subseteq \Ker {\Delta_n^Y}_{|    H_p(Y)},$
  holds if $f: X \longrightarrow Y$ denotes a morphism in $\Top_n$
and $\{\Delta_m^X\}_m$ and $\{\Delta_m^Y\}_m$ denote any transferred $A_\infty$-coalgebra structures on $  H_*(X)$ and $  H_*(Y)$, respectively.
To show that (\ref{eq:fKer_inKer}) holds, notice that any continuous function $f: X \longrightarrow Y$ induces maps of differential graded coalgebras at the singular chain level $  C_*(f): C_*(X) \longrightarrow C_*(Y)$. Applying Thm. \ref{thm:HTT} once with $M\cong C_*(X)$ and once with $M\cong C_*(Y)$, we can infer that $C_*(f)$, in turn, induces a morphism of
$A_\infty$-coalgebras
$$ \{F_{(m)}\}_m : (  H_*(X), \{\Delta_n^X\}_n) \longrightarrow (  H_*(Y), \{\Delta_n^Y\}_n)$$
where $F_{(1)} =   H_*(f)$.

The key point here is that if $X, Y  \in \Top_n$, then the identity MI$(n)$
in Def. \ref{DEF:Morph-A_coalg}
becomes
$ \De^Y_n   H_*(f) =
{  H_*(f)}^{\otimes n} \De_n^X, $
and thus, the inclusion in (\ref{eq:fKer_inKer}) does hold.

Now that we have checked that 
$\kappa_{n,p}(f) \colon \kappa_{n,p}(X) \rightarrow \kappa_{n,p}(Y)$ is well defined, notice that $\kappa_{n,p}(f)$ is the restriction of the map $H_p(f)$, and therefore, the functoriality of $\kappa_{n,p}$ follows from that of the homology functor $H_p: \Top \longrightarrow \Vect$.
\end{proof}

\begin{remark}
\label{rmk:Top_n_is_the_largest_4_functoriality}
In general, one cannot expect to find a larger category in the family $\{\Top_m\}_{m\in \mathbb Z \cup \{\infty\}, m>1}$ for which 
$$\kappa_{n,p}: \Top_m \longrightarrow \Vect$$ would be functorial, as explained next.

Consider we try to redefine 
$\kappa_{n,p}$ as in Def. \ref{Def:Functor_n_p} but on a general $Top_m$,
$$\kappa_{n,p}: \Top_m \longrightarrow \Vect.$$
Thm. \ref{thm:Functoriality} shows that
\begin{itemize}
    \item [(a)] $\kappa_{n,p}: \Top_m \longrightarrow \Vect$ defines a functor, and
    \item [(b)] $\dim \kappa_{n,p}(X)$
does not depend on the choice of $A_\infty$-coalgebra, for any $X \in \Top_m$,
\end{itemize}
for any values $m\geq n \geq2$,
but it does not guarantee these two properties to hold if $2\leq m < n$.
Actually, we should not expect to have such properties to hold for $m<n$ in general, as counterexamples such as Ex. \ref{ex:Top_n_is_the_largest_4_functoriality} show.
\end{remark}

In this example, we use adaptations of the definitions of $\Top_m$ and $\kappa_{n,p}$ so that all the $A_\infty$-coalgebras considered in their definitions are
transferred $A_\infty$-coalgebras on reduced rational homology.
With these reduced versions of $\Top_m$ and $\kappa_{n,p}$,
we will now recall an example for which 
$\dim \kappa_{3,7}(X)$
does depend on the choice of $A_\infty$-coalgebra if $X\in\Top_2-\Top_3$, failing to satisfy property (b),
and we will extend this example to exhibit a case in which 
$$\kappa_{3,7}: \Top_2 \longrightarrow \Vect$$
does not define a functor.

\begin{example}
\label{ex:Top_n_is_the_largest_4_functoriality}
Let us denote a wedge of a complex projective plane and a 7-sphere as $X$, and let us omit the rational coefficients from the notation to simply keep
$\widetilde{H}_*(X) \coloneqq \widetilde{H}_*(\mathbb C P^2 \vee \mathbb S^7; \mathbb Q)$.
Ex. 1 in \cite{Belchi17} presents two transferred $A_\infty$-coalgebra structures $\{\Delta^V_n\}_n$ and $\{\Delta^W_n\}_n$ on $\widetilde{H}_*(X)$. 
From the computations in \cite[Ex. 1]{Belchi17},
it is clear that $X \in \Top_2-\Top_3$. An intuitive reason why this happens is that the cup product of the complex projective plane is non-trivial.
The $A_\infty$-structure $\{\Delta^V_n\}_n$ satisfies
$$ \dim \Ker {\Delta^V_3}_{|\widetilde{H}_7(X)} = \dim \widetilde{H}_7(X) = 1,$$
whereas $\{\Delta^W_n\}_n$ satisfies
$$ \dim \Ker {\Delta^W_3}_{|\widetilde{H}_7(X)} = 0.$$
This directly shows an example for which 
$\dim \kappa_{3,7}(X)$
does depend on the choice of $A_\infty$-coalgebra if $X\in\Top_2-\Top_3$, failing to satisfy property (b).

Let us now attach a cell to $X$ by setting $Y=X\vee\mathbb S^1.$
Notice that the homology of $X$ and $Y$ only differ in degree 1, where we have
$\widetilde{H}_1(X)=0$ and $\widetilde{H}_1(Y)\cong\mathbb Q$.
With the techniques from \cite[III.3.(6)]{Tanre83}, it is easy to see that we can extend the 
$A_\infty$-coalgebra $\left(\widetilde{H}_*(X), \{\Delta^W_n\}_n\right)$
to a transferred
$A_\infty$-coalgebra $\left(\widetilde{H}_*(Y), \{\Delta_n\}_n\right)$,
by simply setting
${\Delta_n}_{|\widetilde{H}_p(Y)} \equiv {\Delta^W_n}_{|\widetilde{H}_p(X)}$ for $p\neq 1$ and
${\Delta_n}_{|\widetilde{H}_1(Y)} \equiv0.$
In particular, 
$$ \dim \Ker {\Delta_3}_{|\widetilde{H}_7(Y)} = \dim \Ker {\Delta^W_3}_{|\widetilde{H}_7(X)} = 0.$$
This leads us to the following situation.
Let us denote by 
$$\iota \coloneqq X \longrightarrow Y$$
the inclusion map of $X$ within $Y$.
If $\alpha$ denotes the generator of $\widetilde{H}_7(X) \cong \mathbb Q,$
then its image by the map induced in homology,
$\widetilde{H}_7(\iota)(\alpha)$, generates $\widetilde{H}_7(Y) \cong \mathbb Q.$
In particular, $0\neq \widetilde{H}_7(\iota)(\alpha) \notin \Ker {\Delta_3}_{|\widetilde{H}_7(Y)},$ and therefore,
$$\widetilde{H}_7(\iota)\left(\Ker {\Delta_3}_{|\widetilde{H}_7(X)}\right) \not\subset \Ker {\Delta_3}_{|\widetilde{H}_7(Y)}.$$
This means we cannot define $\kappa_{3,7}$ on functions of $Top_2-Top_3$.
In particular, $$\kappa_{3,7}: \Top_2 \longrightarrow \Vect$$
does not define a functor.
\end{example}

From this moment on, we assume we have fixed a choice of a functor
$\kappa_{n,p}: \Top_n \longrightarrow \Vect$ as in Def. \ref{Def:Functor_n_p}.
We will next see how the existence of barcodes in $A_\infty$ persistent homology and their stability follow from the functoriality of $\kappa_{n,p}$.

\begin{corollary}
\textbf{(Structure theorem of $\bm{A_\infty}$ persistent homology in $\Top_n$)}
\label{cor:Barcode_Ainfty}
Fix integers $p\geq 0$ and $n \geq 1$. Let ${X}_* \colon \mathbb R \longrightarrow \Top_n$ be a persistence space such that $\dim_\mathbb{F} H_p(X_t) < \infty$ for all $t \in \mathbb R$.
Then the $A_\infty$ persistent homology module $\kappa_{n, p} {X}_*$ 
decomposes uniquely (up to isomorphism)  
into interval persistence modules $C(I)$,
$$\kappa_{n, p} {X}_*\cong\bigoplus_{I\in B\left(\kappa_{n, p} {X}_* \right)}\Cfunc(I),$$
where $B\left(\kappa_{n, p} {X}_* \right)$ is a multiset of intervals of the form $[a,b)$ for some $a\in \mathbb R$, $b\in (a, +\infty] \subseteq \mathbb{R}\cup\{+\infty\}$. We call $B\left(\kappa_{n, p} {X}_* \right)$ the  \textbf{${\bm\Delta_{n,p}}$-barcode} of ${X}_*$.
\end{corollary}

\begin{proof}
Since
$\kappa_{n, p}(X_t)$ is a vector subspace of $H_p(X_t)$,
we have
$$\dim_\mathbb{F} \kappa_{n, p}(X_t) \leq \dim_\mathbb{F} H_p(X_t) < \infty$$
for all $t \in \mathbb R$, by assumption.
Hence, $\kappa_{n, p} {X}_*$ forms a p.f.d. persistence module and Thm.~\ref{thm:Barcode_decomposition} guarantees that we can uniquely decompose it as a direct sum of interval persistence modules.
\end{proof}

The barcode decomposition result in Cor. \ref{cor:Barcode_Ainfty} deals with persistence spaces of the form ${X}_* \colon \mathbb R \longrightarrow \Top_n$.
An analogous theorem was proved in \cite{Belchi17}
for persistence spaces
${X}_* \colon \mathbb Z \longrightarrow \mathcal C$
indexed by the integers and valued in a category $\mathcal C$ potentially larger than $\Top_n$ ($\Top_n \subset \mathcal C \subset \Top$).
When the considered persistence spaces have the form ${X}_* \colon \mathbb Z \longrightarrow \Top$,
the inclusion in (\ref{eq:fKer_inKer}) fails. This stops
the analogue of $\kappa_{n,p}({X}_*)$ from forming a persistence module, and one needs to resort to zigzag modules. The structure and interpretation of the corresponding $A_\infty$ persistence zigzag modules was studied in detail in \cite{Belchi-Murillo15}.

To state and prove the rest of the results in this paper, we will fix the following notation and assumptions:

\begin{assumptions*}
All the sublevel sets 
of a real-valued functions $f \colon X \longrightarrow \bR$, 
have finite-dimensional $p^\text{th}$ homology, \emph{i.e.,} $\dim_{\mathbb F} H_p\left( f^{-1} [-\infty, t)\right) < \infty$.
\end{assumptions*}

\begin{notation*}
\begin{itemize}
    \item $d_B$ denotes the bottleneck distance (Def. \ref{def:bottleneck_distance}).
    \item $d_\infty$ denotes the $l^\infty$ distance in (\ref{eq:l-infinity_distance}).
    \item $|| \cdot||_\infty$ denotes the supremum norm.
    \item $B_{n, p}(f) \coloneqq B\left(\kappa_{n, p}\SS(f)\right)$ denotes the $\Delta_{n,p}$-barcode of the sublevel-set filtration of a continuous real-valued function $f \colon X \longrightarrow \mathbb{R}$.
    \item $d^X:M \longrightarrow \mathbb{R}$ denotes the real-valued function defined by $d^X(y) = d(y,X)$, for all $y\in M$, given any closed subspace $X$ of a metric space $(M, d)$.
    \item $X^{+\delta}$ denotes the (open) \textbf{$\delta$-thickening} of $X$, $\left(d^X\right)^{-1}\left(-\infty, \delta\right),$ for $\delta>0.$
    \item $X^{+\delta]}$ denotes the (closed) \textbf{$\delta$-thickening} of $X$, $\left(d^X\right)^{-1}\left(-\infty, \delta\right],$ for $\delta\geq0.$
    \item $B_{n,p}(X)$ denotes the barcode $B_{n, p}(d^X)$.
\end{itemize}
\end{notation*}

As we will mention after Cor.~\ref{cor:stability_Hausdorff}, $B_{n, p}(X)$ can be interpreted as the $\Delta_{n,p}$-barcode of the \Cech ech-complex filtration of $X$.

Let us start with a toy example in which two persistent spaces result in the same barcode within classical persistent homology but result in barcodes at an infinite bottleneck distance away from each other within $A_\infty$ persistent homology. This example illustrates the superior discriminatory power of $A_\infty$ persistent homology and helps understanding what the barcodes in $A_\infty$-persistence encode.

\begin{example}
\label{ex:barcodes_in_PH_and_Ainfty}
Let us work with reduced homology. 
We start by defining the category $\widetilde{\Top}$ as the analogue of $\Top_n$ (Def. \ref{Def:Top_m}) in this setting: for $n\in\mathbb{Z}\cup\{+\infty\}$, $n>1$, let $\widetilde{\Top}_n$ denote the category whose objects are topological spaces $X$ such that $\Delta_m = 0$, for all $m < n $,
where $\{\Delta_m\}_{m\geq1}$ denotes any transferred
$A_\infty$-coalgebra structure on the reduced homology $\widetilde{H}_*(X;\mathbb F)$,
and where the morphisms are continuous maps.

Let $X_*, Y_*$ be two persistence spaces consisting of thickeniing filtrations of two point clouds $P$ and $Q$, in the sense that
$$
X_t \coloneqq P^{+t]}
\hspace{5mm} \text{and} \hspace{5mm} 
Y_t \coloneqq Q^{+t]}, \hspace{5mm}\text{for all } t\in\mathbb R.$$
Assume that $P$ and $Q$ are sampled from a torus $\mathbb T$ and a wedge of spheres $\mathbb S^2 \vee \mathbb S^1 \vee \mathbb S^1$, respectively, in such a way that the non-trivial degree 2 homology of the persistence spaces appears at time $t_0$ and vanishes at time $t_1$ in both cases, 
\emph{i.e.,}
$$H_2(X_t)=H_2(Y_t)=0,  \hspace{5mm}\text{for all } t\in (-\infty, t_0)\cup [t_1, +\infty)$$
and so that the homotopy types of the steps in-between are known:
$$
X_t \simeq \mathbb T
\hspace{5mm} \text{and} \hspace{5mm} 
Y_t \simeq \mathbb S^2 \vee \mathbb S^1 \vee \mathbb S^1, \hspace{5mm}\text{for all } t\in\mathbb [t_0, t_1).$$
where all connecting maps are homotopic to the identity.
Denote by $B_p(X_*)$ and $B_p(Y_*)$ the barcodes describing the evolution of $\widetilde{H}_p(X_t)$ and $\widetilde{H}_p(Y_t)$, respectively.
Recall that
$$
\widetilde{H}_p(\mathbb T)
\cong
\widetilde{H}_p(\mathbb S^2 \vee \mathbb S^1 \vee \mathbb S^1)
\cong
\begin{cases}
\begin{array}{lcl}
\mathbb F^2, & \quad & p = 1\\
\mathbb F, & \quad & p = 2\\
0, & \quad & p \neq 1,2.
\end{array}
\end{cases}
$$
In particular, both $B_2(X_*)$ and $B_2(Y_*)$ consist of a single interval $[t_0, t_1)$. Hence,
$$ d_B \left(B_2(X_*), B_2(Y_*)\right)=0,$$
and classical persistence in degree 2 does not tell the two point clouds apart.

Now let us also denote by $\alpha^*$ and $\beta^*$ the two generators of 
$\widetilde{H}^1(\mathbb T)$ and let $\gamma^*$ be the generator of $\widetilde{H}^2(\mathbb T)$.
It is well known that we can choose these cohomology generators so that the cup product relates them via the equality
$\alpha^* \smile \beta^* = \gamma^*.$
Dually, if we denote by $\alpha$ and $\beta$ the two generators of 
$\widetilde{H}_1(\mathbb T)$ and let $\gamma$ be the generator of $\widetilde{H}_2(\mathbb T)$, then 
$$\Delta \gamma = \alpha \otimes \beta.$$
Hence, any transferred $A_\infty$-coalgebra 
$\left(\widetilde{H}_*(\mathbb T), \{\Delta_n\}_n \right)$
will have
\begin{equation}
\label{eq:Delta2_no_0}
    \Delta_2(\gamma)\neq0,
\end{equation}
and thus
$$X_t \in \widetilde{Top}_2\setminus\widetilde{Top}_3,$$
for all $t\in [t_0, t_1).$ Therefore, as long as we only compare $\Delta_{n,p}$-barcodes of $X_*$ for $n\leq2$, we can guarantee that stability results such as Cor. \ref{cor:Stability_A_infty} hold.
Denote by $B_{2,2}(X_*)$ and $B_{2,2}(Y_*)$ the $\Delta_{2,2}$-barcode describing the evolution of $\Ker {\De_2}_{| \widetilde{H}_2(X_t)}$ and $\Ker {\De_2}_{| \widetilde{H}_2(Y_t)}$, respectively.
\noindent Eq. (\ref{eq:Delta2_no_0}) shows that
$$\Ker {\De_2}_{| \widetilde{H}_2(X_t)} = 0 \subsetneq \widetilde{H}_2(X_t),$$
for all $t\in [t_0, t_1).$
Hence, $B_{2,2}(X_*)$ is empty - it consists of no intervals.

On the other hand, using the same reasoning on $\mathbb S^2 \vee \mathbb S^1 \vee \mathbb S^1$, any transferred $A_\infty$-coalgebra 
$\left(\widetilde{H}_*(\mathbb S^2 \vee \mathbb S^1 \vee \mathbb S^1), \{\Delta_n\}_n \right)$
will have $\Delta_2 = 0,$
which is, 
$$\Ker {\De_2}_{| \widetilde{H}_2(Y_t)} = \widetilde{H}_2(Y_t),$$
for all $t\in [t_0, t_1).$
Hence, $B_{2,2}(Y_*)=B_2(Y_*)$ still consists of a single interval 
$[t_0, t_1)$ and we conclude that
$$ d_B \left(B_{2,2}(X_*), B_{2,2}(Y_*)\right)=\frac{t_1-t_0}{2}.$$
The more the degree 2 homology persist in the filtrations (\emph{i.e.,} the greater the difference $t_1-t_0$ is), the greater the bottleneck distance between the $A_\infty$-persistence barcodes $B_{2,2}(X_*)$ and $B_{2,2}(Y_*)$ is too.
\end{example}

We now prove the first result on the stability of $A_\infty$ persistent homology by providing a generalization of Thm.~\ref{thm:stability_classical}.

\begin{corollary}
\textbf{(Stability of $\bm{A_\infty}$ persistent homology for functions)}
\label{cor:Stability_A_infty}
Let $n>0$ be an integer and let $f \colon X \longrightarrow \mathbb{R}$ and $g \colon Y \longrightarrow \mathbb R$ be two continuous maps. 
Assume that all sublevel-sets of $f$ and $g$ are in $\Top_n$.
Then, for all $p\geq 0$,
the bottleneck distance between the $\Delta_{n,p}$-barcodes of $f$ and $g$ is bounded above by the $l^\infty$ distance between the functions:
$$ d_B \left( B_{n, p}(f), B_{n, p}(g) \right) \leq d_\infty(f, g).$$
In particular, if $X = Y$, then
$$ d_B \left( B_{n, p}(f), B_{n, p}(g) \right) \leq || f - g ||_\infty.$$
\end{corollary}

\begin{proof}
The assumption of each $f^{-1} [-\infty, t)$ being in $\Top_n$ turns the sublevel-set filtration of $f$ into a persistence space of the form 
$\SS(f) \colon \mathbb R \longrightarrow \Top_n$. Since $\kappa_{n,p}$ is a functor (Thm.~\ref{thm:Functoriality}) and $\dim_{\mathbb F} H_p\left( f^{-1} [-\infty, t)\right) < \infty$ for all $t$ by assumption,
the composition $\kappa_{n,p} \SS(f)$ is a p.f.d. persistence module.
Thm.~\ref{thm:Barcode_decomposition} guarantees the existence and uniqueness of the $\Delta_{n,p}$-barcode $B_{n, p}(f) \coloneqq B \left( \kappa_{n,p} \SS(f) \right)$. The same holds for $g$, and Thm.~\ref{thm:Isometry_theorem} then shows that
$$d_B \left( B\left(\kappa_{n, p}\SS(f)\right), B\left(\kappa_{n, p}\SS(g)\right) \right) \leq d_I \left( \kappa_{n, p}\SS(f), \kappa_{n, p}\SS(g) \right),$$
where $d_I$ denotes the interleaving distance.
We then use Thm.~\ref{thm:Functors_are_1Lipschitz_Real_case} to exploit the functoriality of $\kappa_{n,p}$ and conclude that
$$d_I \left( \kappa_{n, p}\SS(f), \kappa_{n, p}\SS(g) \right) \leq
d_I \left( \SS(f), \SS(g) \right).$$
Thm.~\ref{thm:SublevelSet_1Lipschitz} asserts that the sublevet-set functor is 1-Lipschitz. Hence:
$$d_I \left( \SS(f), \SS(g) \right) \leq
d_\infty(f, g).$$
These 3 inequalities together show that 
$$ d_B \left( B_{n, p}(f), B_{n, p}(g) \right) \leq d_\infty(f, g).$$
In particular, if $X = Y$, then, since the identity $\Phi \colon X \longrightarrow Y$ is a homeomorphism, 
$$ d_\infty(f, g) \leq || f - g ||_\infty,$$ and the second claim in Cor.~\ref{cor:Stability_A_infty} holds as well.
\end{proof}

\begin{definition}
\label{def:Hausdorff_distance}
The \textbf{Hausdorff distance} between two non-empty subsets $X, Y$ of a metric space $(M, d)$ is defined as
$$d_H(X,Y) = \inf\{\delta \geq 0\,;\ X \subseteq Y^{+\delta]} \ \mbox{and}\ Y \subseteq X^{+\delta]}\}.$$
\end{definition}

By definition of the Hausdorff distance $d_H$, we have that $||d^X - d^Y||_\infty = d_H(X,Y)$.
This leads to the following straightforward corollary of Cor.~\ref{cor:Stability_A_infty}:

\begin{corollary}
\textbf{(Stability of $\bm{A_\infty}$ persistent homology for metric spaces)}
\label{cor:stability_Hausdorff}
Let $n\geq 1$ be an integer.
Let $X$ and $Y$ be closed subspaces of a metric space $M$
such that the thickenings $X^{+\tau]}$ and $Y^{+\tau]}$ are in $\Top_n$ for all $\tau\geq 0$. 
Then, for all $p \geq 0$,
the bottleneck distance between the $\Delta_{n,p}$-barcodes $B_{n,p}(X)$ and $B_{n,p}(Y)$ is bounded above by the Hausdorff distance between the spaces $X$ and $Y$:
$$ d_B(B_{n,p}(X), B_{n,p}(Y)) \leq 
d_H(X,Y).$$
\end{corollary}
    
The condition $X^{+\tau]}, Y^{+\tau]}\in\Top_n$ for all $\tau\geq 0$ is equivalent to assuming
all sublevel sets of $d^X$ and $d^Y$ to be in $\Top_n$.
Hence, Cor. \ref{cor:stability_Hausdorff} follows directly from applying Cor. \ref{cor:Stability_A_infty} to $f=d^X, g=d^Y$.

For computational purposes, one tends to work with simplicial or cubical complexes.
Let $X$ be a finite subset of a metric space $(M, d)$. For any $\epsilon > 0$, the 
\Cech ech complex $\cech_\epsilon (X)$ is a simplicial complex defined as the nerve of the set of open balls of radius $\epsilon$ centered at all points in $X$\footnote{Note that $\cech_\epsilon (X)$ is sometimes defined using the radius $\epsilon/2$.}.
The Nerve Lemma \cite[Corollary 3 in Section 9]{Borsuk48} shows that $\cech_\epsilon (X)$ is homotopy equivalent to the sublevel set $X^{+\epsilon} = {\left(d^X\right)}^{-1}(-\infty, \epsilon)$, provided that all intersections of balls are either empty or contractible.
Hence, Cor.~\ref{cor:stability_Hausdorff} can be interpreted as a result on the stability of the $A_\infty$ persistent homology of the \Cech ech-complex filtration.

We finally come to extrapolating topological properties of a closed subspace $X$ of a metric space $M$ from a finite set $P$ of points which may have been inaccurately sampled from $X$.
For smooth manifolds $X$ in low dimensions, some methods \cite{Amenta-Bern99,Dey-Goswami04, Niyogi-Smale-Weinberger08} can reconstruct $X$ from $P$ so that one can estimate the Betti numbers of $X$.
V. Robins \cite{Robins99}, V. de Silva and G. Carlsson \cite{Carlsson-de_Silva04} started estimating these Betti numbers via persistent homology instead. One of the advantages of this approach is that it can be used for smooth and non-smooth spaces $X$ of all dimensions, and it avoids having to choose an optimal thickening amount which may sometimes be impossible to find. Later on, stability would play a crucial role in this homology estimation task \cite{Cohen_Steiner-Edelsbrunner-Harer07, Blumberg-Gal-Mandell-Pancia14, FasyEtAl14}.
F. Chazal and A. Lieutier \cite{Chazal-Lieutier05} were the first to estimate topological information of $X$ not captured by the homology groups. Namely, they approximated the fundamental group $\pi_1(X)$. More on fundamental group of point clouds can be found in \cite{Brendel-Dlotko-EllisEtAl15}.
The current work goes further in this direction by showing that we can also estimate $A_\infty$ information of $X$, as we will see in Cor.~\ref{cor:ainfty-inference}.

We now define an analogous concept to that of homological critical value (Def.~\ref{def:homological_critical_value}) for $A_\infty$-structures.

\begin{definition}
Given a continuous map $f \colon X \longrightarrow \mathbb R$ whose sublevel sets are in $\Top_n$ and
a functor $\kappa_{n,p}$ defined as in Def.~\ref{Def:Functor_n_p}, we say that a real number $a\in\mathbb R$ is a \textbf{$\bm{\Delta_{n,p}}$ critical value} of $f$ if the map 
$$\kappa_{n,p}\mathcal S (f)(a-\epsilon) \longrightarrow \kappa_{n,p}\mathcal S (f)(a+\epsilon)$$
induced in homology by the inclusion
$$ f^{-1}(-\infty, a-\epsilon] \xhookrightarrow{} f^{-1}(-\infty, a+\epsilon] $$
is not an isomorphism for all sufficiently small $\epsilon > 0$.
We then define the \textbf{$\bm{\Delta_{n,p}}$ feature size} of $X$, denoted by $\Delta_{n,p}fs(X)$, as the infimum over all positive $\Delta_{n,p}$ critical values of the distance function $d^X$.
\end{definition}

Just as the homological feature size used in \cite{Cohen_Steiner-Edelsbrunner-Harer07}, the 
$\Delta_{n,p}$ feature size depends not only on the topology of $X$, but also on its geometry.

Given $n, p$ and a functor $\kappa_{n,p}$ defined as in Def.~\ref{Def:Functor_n_p}, the following result computes $A_\infty$-information of a $\delta$-thickening of $X$ via the $A_\infty$ persistent homology of $P$, where we can think of $P$ as a finite approximation to $X$.

\begin{corollary}\textbf{($\bm{A_\infty}$ inference)}
\label{cor:ainfty-inference}
Let $X, P$ be closed subspaces of a metric space $M$ 
such that the thickenings $X^{+\tau]}$ and $P^{+\tau]}$ are in $\Top_n$ for all $\tau\geq 0$. 
Let $\epsilon>0$ be a real number with $d_H(X,P) <\epsilon< \frac{\Delta_{n,p}fs(X)}{4}$. For all sufficiently small $\delta>0$, the topological invariant $\dim \kappa_{n,p}(X^{+\delta})$
coincides with the number of intervals in the $\Delta_{n,p}$-barcode $B_{n,p}(P)$ which contain the interval $[\epsilon, 3\epsilon].$
\end{corollary}

The lower bound on $\epsilon$ appearing in Cor. \ref{cor:ainfty-inference} describes how accurately $P$ approximates $X$. The better $P$ approximates $X$, the smaller $d_H(X,P)$ is.
The upper bound on $\epsilon$ appearing in Cor. \ref{cor:ainfty-inference}, $\frac{\Delta_{n,p}fs(X)}{4}$, depends on the metrics of $X$ and $M$.

The proof of Cor.~\ref{cor:ainfty-inference} is basically that of \cite[Homology Inference Theorem]{Cohen_Steiner-Edelsbrunner-Harer07}, changing the persistent homology functor by the $A_\infty$ persistent homology functor $\kappa_{n,p}$. We include the proof anyway for the sake of completeness.

\begin{proof}
Let $f, g \colon M \longrightarrow \mathbb R$ be continuous maps whose sublevel sets are all in $\Top_n$. 
Let us define 
$$ F_i \coloneqq \kappa_{n,p} \left( f^{-1} (-\infty, i] \right) = \kappa_{n,p}  \mathcal S \left( f \right) (i) 
\hspace{5mm}\text{and}\hspace{5mm}
G_i \coloneqq \kappa_{n,p} \left( g^{-1} (-\infty, i] \right) = \kappa_{n,p}  \mathcal S \left( g \right) (i)$$
for all $i\in\mathbb R$.
Since $\kappa_{n, p}$ is functorial (Thm. \ref{thm:Functoriality}), we can consider the map
$$\kappa_{n,p}\left(f^{-1} (-\infty, i] \xhookrightarrow{} f^{-1} (-\infty, j]\right)$$ for any real numbers $i<j$, which we will denote by $f_i^j \colon F_i \longrightarrow F_j$.
Define $g_i^j \colon G_i \longrightarrow G_j$ analogously.
Finally, let us set 
$$F_i^j \coloneqq \Image f_i^j
\hspace{5mm}\text{and}\hspace{5mm}
G_i^j \coloneqq \Image g_i^j.$$ This way, $\dim F_i^j$ and $\dim G_i^j$ are precisely the number of intervals which contain $[i, j)$ in the $A_\infty$-barcodes $B_{n,p}(f)$ and $B_{n,p}(g)$, respectively.
If $f=d^X$ or $f=d^P$, we will denote $F_i^j$ by $X_i^j$ or $P_i^j$, respectively.

If $||f-g||_\infty < \epsilon$, then $f^{-1} (-\infty, i] \subseteq g^{-1} (-\infty, i+\epsilon]$ for all $i\in\mathbb R$.
By the functoriality of $\kappa_{n, p}$ (Thm.~\ref{thm:Functoriality}), this inclusion induces a map $\varphi_i \colon F_i \longrightarrow G_{i+\epsilon}$.
Exchanging $f$ and $g$, we also have a map $\psi_i \colon G_i \longrightarrow F_{i+\epsilon}$.
We can fit these maps into the following diagram:
$$\xymatrix{
F_{i-\epsilon} \ar[r]^-{f_{i-\epsilon}^{j+\epsilon}} \ar[d]_-{\varphi_{i-\epsilon}} & F_{j+\epsilon} \\
G_i \ar[r]_-{g_i^j} & G_j \ar[u]_-{\psi_j}
}$$
Since all maps in the diagram are induced by inclusions, the diagram commutes, for any reals $i<j$.
Some diagram chasing shows that 
$$ F_{i-\epsilon}^{j+\epsilon} \subseteq \psi_j\left(G_i^j\right),$$
and hence 
\begin{equation}
\label{eq:homologyInference_dimensions}
\dim F_{i-\epsilon}^{j+\epsilon} \leq \dim G_i^j.
\end{equation}

Now set
$f=d^X, g=d^P, i=\epsilon + \delta$ and $j=3\epsilon+ \delta$.
Since $||d^X-d^P||_\infty=d_H(X,P)$ and we assumed the $d_H(X,P)<\epsilon$ and that all sublevel sets of $d^X$ and $d^P$ are in $\Top_n$, 
we can use (\ref{eq:homologyInference_dimensions}) and rewrite it as
\begin{equation}
\label{eq:homologyInference_dimensions1}
\dim X_{\delta}^{4\epsilon + \delta} \leq \dim P_{\epsilon + \delta}^{3\epsilon+\delta}.
\end{equation}
Analogously, letting $f=d^P, g=d^X$ and $i=j=2\epsilon+\delta$, (\ref{eq:homologyInference_dimensions}) yields
\begin{equation}
\label{eq:homologyInference_dimensions2}
\dim P_{\epsilon+\delta}^{3\epsilon + \delta} \leq \dim X_{2\epsilon + \delta}^{2\epsilon+\delta}.
\end{equation}

Choosing $\delta$ small enough so that $4\epsilon + \delta < \Delta_{n,p}fs(X)$, there are no $\Delta_{n,p}$ critical values of $d^X$ in $[\delta, 4\epsilon + \delta]$.
In particular, $\dim X_{\delta}^{4\epsilon + \delta} = \dim X_{2\epsilon + \delta}^{2\epsilon+\delta},$ and the inequalities in (\ref{eq:homologyInference_dimensions1}) and (\ref{eq:homologyInference_dimensions2}) becomes equalities.
Again, since there are no $\Delta_{n,p}$ critical values of $d^X$ in $[\delta, 4\epsilon + \delta]$,
$\dim X_\delta^{4\epsilon+\delta} = \dim X_\delta^\delta = \dim \kappa_{n,p}(X^{+\delta})$.

Since $\epsilon>0$, $\dim P_{\epsilon + \delta}^{3\epsilon+\delta} = \dim P_{\epsilon}^{3\epsilon}$ for every $\delta>0$ small enough.
Therefore, for sufficiently small $\delta>0$, 
$$\dim \kappa_{n,p}(X^{+\delta}) = \dim P_{\epsilon + \delta}^{3\epsilon+\delta} = \dim P_{\epsilon}^{3\epsilon},$$
concluding the proof.
\end{proof}

Actually, the proof of Cor. \ref{cor:ainfty-inference} clearly shows that we do not need to assume that all sublevel sets of $d^X$ and $d^P$ are in $\Top_n$. Rather, it is enough to assume that 
$X^{+\tau]}\in\Top_n$ for all $\tau\leq 4\epsilon + \delta$ and that 
$P^{+\tau]}\in\Top_n$ for all $\tau\leq 3\epsilon + \delta$.

In most real world situations and for small enough values of $\delta>0$, $X$ is a retract of $X^{+\delta}$ and therefore the homotopy types of $X$ and $X^{+\delta}$ coincide. Thus, we can see Cor.~\ref{cor:ainfty-inference} as estimating topological properties of $X$ from a (possibly finite) closed subset $P\subseteq M$ approximating $X$.

Note that when we focus on the second operation $\Delta_2$ on an $A_\infty$-coalgebra $\left(H_*(X), \{\Delta_n\}_n \right)$
 (such as in Ex. \ref{ex:barcodes_in_PH_and_Ainfty}), or equivalently, on the cup product in cohomology (which is the second operation $\mu_2=\smile$ on an $A_\infty$-algebra $\left(H^*(X), \{\mu_n\}_n \right)$), then all results in this paper hold without the need to restrict to a category $\Top_n \subseteq \Top$ for $n>2$ and instead, we can work directly with the category of topological spaces $\Top$. In particular, this paper proves the stability of the persistence of cup product with minimal restrictions.

\section*{Acknowledgements}
F. Belch\'i was partially supported by the EPSRC grant EPSRC EP/N014189/1 \emph{(Joining the dots)} to the University of Southampton and by the Spanish State Research Agency through the Mar\'ia de Maeztu Seal of Excellence to IRI (MDM-2016-0656).
A. Stefanou was partially supported both by the National Science Foundation through grant NSF-CCF-1740761 TRIPODS TGDA@OSU and by the Mathematical Biosciences Institute at the Ohio State University.

\bibliographystyle{plain}  
{
\footnotesize
\bibliography{mybib_Kiko}{}
}
\end{document}